\documentclass[10pt,reqno]{amsart}
\usepackage{amssymb,amscd,url} 
\usepackage[all]{xy}  
\usepackage{color}

\theoremstyle{plain}
\newtheorem{theorem}{Theorem}

\newtheorem{corollary}[theorem]{Corollary}

\newtheorem{claim}[theorem]{Claim}

\theoremstyle{definition}

\begin{document}

\title{On Arveson's Boundary Theorem} 
\author[K.\ Hasegawa]{Kei HASEGAWA$\,^{1}$} 
\address{
$\,^{1}$Department of Mathematics, Kyoto university, Kitashirakawa-Oiwakecho, Sakyo-ku, Kyoto, 606-8502, Japan
}
\email{khasegawa@math.kyoto-u.ac.jp}
\author[Y.\ Ueda]{Yoshimichi UEDA$\,^{2}$}
\address{
$\,^{2}$Graduate School of Mathematics, 
Nagoya University, 
Furocho, Chikusaku, Nagoya, 464-8602, Japan
}
\email{ueda@math.nagoya-u.ac.jp}
\thanks{YU was partly supported by Grant-in-Aid for Scientific Research (C) JP24540214 and (B) JP18H01122. 
KH was partly supported by Grant-in-Aid for JSPS Research Fellow JP18J00453.}

\date{Oct.\ 25th, 2018}
\maketitle
\begin{abstract}
This short note aims to give an insight to Arveson's {\it boundary} theorem by means of non-commutative Poisson {\it boundaries} and its applications. 
\end{abstract}

\allowdisplaybreaks{

\section{Introduction} 

Let $A$ be a subset of a unital $C^*$-algebra $B$ such that $A$ generates $B$ as unital $C^*$-algebra (we denote this by $B = C^*(A)$ in what follows). An irreducible $*$-representation $\pi$ of $B$ is called a {\it boundary representation for $A$} if the $\pi$ itself is only possible UCP extension of the restriction map $\pi\!\upharpoonright_A$ to $B$. The unitary equivalence classes of those (though they have a logical issue) should be regarded as the possible `non-commutative Choquet boundary' for $A$. In his seminal work \cite{Arveson:ActaMath6972}, William Arveson introduced this notion and investigated it in many concrete examples. One of the highlights there is the next theorem together with its application to unitary equivalence between irreducible compact operators.    

\medskip\noindent
{\bf Arveson's boundary theorem.} {\it For any subset $A \subset B(\mathcal{H})$ with $C^*(A)$ irreducible, the following are equivalent:  
\begin{itemize} 
\item[(1)] The identity representation $\mathrm{id}$ of $B(\mathcal{H})$ on $\mathcal{H}$ is only possible UCP extension of the identity map $\mathrm{id}_A$ on $A$ to $B(\mathcal{H})$. 
\item[(2)] The restriction of the quotient map $\pi \colon B(\mathcal{H}) \rightarrow B(\mathcal{H})/K(\mathcal{H})$ by the compact operators $K(\mathcal{H})$ to the linear span $S$ of $A \cup A^*$ is not completely isometric. 
\end{itemize}
Moreover, each of items (1) and (2) is equivalent to
\begin{itemize}
\item[(3)] {\it $K(\mathcal{H}) \subseteq C^*(A)$ and $\mathrm{id}_{C^*(A)}$ is a boundary representation for $A$.} 
\end{itemize}}

\medskip
Another notion whose name contains `boundary' in the non-commutative analysis based on operator algebras is that of non-commutative Poisson boundaries introduced by Masaki Izumi \cite{Izumi:AdvMath02}, and it has no theoretic connection with Arveson's boundary theorem. Nevertheless, this note illustrates how this relatively new notion of boundaries provides a rather elementary and straightforward exposition of Arveson's boundary theorem in a slightly more general framework than itself. The proof below is closer to Arveson's original proof (which uses the second dual $B(\mathcal{H})^{**}$) than Davidson's one \cite{Davidson:PAMS81}, and moreover not `dilation theoretic' (though an example of non-commutative Poisson boundary is given by `Toeplitz dilations', see \cite[Appendix]{Izumi:JOT12}, which was observed by Arveson). Our proof is motivated by \cite[Remark 2.4]{Arveson:MathScand10} and Farenick's exposition \cite{Farenick:LAA11}, both of which deal with only the finite dimensional setting, and ours only needs standard facts (Arveson's extension theorem, Choi's technique of multiplicative domains, and the Choi--Effros $C^*$-algebra structure), all of which can be found in \cite{Paulsen:Book}. 

\medskip\noindent
{\bf Acknowledgements.} 
The second-named author would like to thank Professor Shigeru Yamagami for giving him an opportunity to give an intensive course in Dec.\ 2015, which made the second-named author find an initial observation to this note. The authors would also like to thank Professor Ken Davidson for informing the second-named author of another short proof to Arveson's boundary theorem itself.   

\section{Arveson's Boundary Theorem by Poisson Boundary} 

The main implication of Arveson's boundary theorem is item (2) $\Rightarrow$ item (1), because the converse one is straightforward with the help of Arveson's extension theorem. Thus, we will explain only the main implication in our way. By item (2), $\pi$ is not injective on $C^*(A)$. Thus, we have $C^*(A) \cap K(\mathcal{H}) \neq \{0\}$, and hence $K(\mathcal{H}) \subseteq C^*(A)$ thanks to the irreducibility of $C^*(A)$ (see e.g. \cite[Corollary I.10.4]{Davidson:Book}). We choose any UCP extension $\varphi$ of $\mathrm{id}_A$ to $B(\mathcal{H})$, and the desired conclusion $\varphi = \mathrm{id}$ immediately follows from a more general theorem below with $B = B(\mathcal{H})$ and $J = K(\mathcal{H})$. Here, recall that an ideal $J$ of a $C^*$-algebra $B$ is said to be \emph{essential} if $J$ has the non-zero intersection with any other non-zero ideal, or equivalently, $x J = \{0\}$ implies $x = 0$ for every $x \in B$. Note that $K(\mathcal{H})$ is a simple essential ideal of $B(\mathcal{H})$.  

\begin{theorem}\label{T} Let $B$ be a unital $C^*$-algebra. Let $\varphi \colon B \to B$ be a UCP map, and $B^\varphi$ the operator system consisting of all the fixed points under $\varphi$. Assume that there exists a simple essential ideal $J$ of $B$ such that $J \subseteq C^*(B^\varphi)$. Then the following are equivalent: 
\begin{itemize} 
\item[(i)] $\varphi = \mathrm{id}$. 
\item[(ii)] The restriction of the quotient map $\pi : B \to B/J$ to $B^\varphi$ is not completely isometric. 
\end{itemize}
\end{theorem}
\begin{proof}
Consider the normal extension of $\varphi$ to $M := B^{**}$, the second dual of $B$. We still denote the normal extension by the same symbol $\varphi : M \to M$. Then it is easy to see that $\psi := \lim_{m\to\omega}\frac{1}{m+1}\sum_{k=0}^{m} \varphi^k$ (in the BW-topology) becomes a UCP idempotent from $M$ onto the fixed points $M^\varphi$ such that 
\begin{align} 
&\varphi\circ\psi = \psi\circ\varphi = \psi, \label{Eq1}\\
&\psi\!\upharpoonright_{M^\varphi} = \mathrm{id}\!\upharpoonright_{M^\varphi}. \label{Eq2}
\end{align} 
(See e.g.\ \cite[Proposition 5.2]{Arveson:JFA07}.) The operator system $M^\varphi$ becomes a $C^*$-algebra (actually, a von Neumann algebra) with the Choi--Effros product $x\circ y := \psi(xy)$ for $x, y \in M^\varphi$ and the $C^*$-algebra $H^\infty(M,\varphi):=(M^\varphi,\circ)$ (which is $M^\varphi$ as an operator system but equipped with the Choi--Effros product $\circ$ instead of the original product) is called the {\it non-commutative Poisson boundary} for $\varphi: M \to M$. Remark that the identity map $M^\varphi \to H^\infty(M,\varphi)$ is a complete order isomorphism (see the proof of \cite[Theorem 15.2]{Paulsen:Book}). We denote by $\mathrm{mult}(\psi)$ the multiplicative domain of the UCP map $\psi \colon M \to H^\infty(M,\varphi)$. The multiplicative domain $\mathrm{mult}(\psi)$ is known to be a unital $C^*$-subalgebra of $M$, and it is also known that $\psi$ becomes multiplicative on it. (See e.g.\ \cite[Theorem 3.18]{Paulsen:Book}). Hence the restriction of $\psi$ to $\mathrm{mult}(\psi)$ gives a unital $*$-homomorphism from $\mathrm{mult}(\psi)$ into $H^\infty(M,\varphi)$. Here is a key (but rather simple) claim. 

\begin{claim}\label{C2} $M^\varphi \subseteq \mathrm{mult}(\psi)$ and $\varphi(\mathrm{mult}(\psi)) \subseteq \mathrm{mult}(\psi)$. 
\end{claim}

\begin{proof} We are dealing with the Choi--Effros $C^*$-algebra structure on the range $M^\varphi$ of the UCP map $\psi$. By Eq.\eqref{Eq2} we have $\psi(x^* x) = x^*\circ x = \psi(x)^*\circ\psi(x)$ for every $x \in M^\varphi$; hence $M^\varphi \subseteq \mathrm{mult}(\psi)$. For any $y \in \mathrm{mult}(\psi)$ we have $\psi(y^* y) = \psi(y)^*\circ\psi(y)$ and thus
\begin{align*}
\psi(\varphi(y)^*\varphi(y)) 
&\leq 
\psi(\varphi(y^* y)) 
\overset{\text{Eq.\eqref{Eq1}}}{=} 
\psi(y^* y) 
= 
\psi(y)^*\circ\psi(y) 
\overset{\text{Eq.\eqref{Eq1}}}{=}  
\psi(\varphi(y))^*\circ\psi(\varphi(y)) \\
&= 
\psi(\psi(\varphi(y))^*\psi(\varphi(y))) 
\leq 
\psi(\psi(\varphi(y)^*\varphi(y))) 
\overset{\psi^2 = \psi}{=}  
\psi(\varphi(y)^*\varphi(y)),
\end{align*}
where the Schwarz inequality is used twice. Hence $\psi(\varphi(y)^*\varphi(y)) = \psi(\varphi(y))^*\circ\psi(\varphi(y))$, that is, $\varphi(\mathrm{mult}(\psi)) \subset \mathrm{mult}(\psi)$.  
\end{proof}  
  
Since item (i) $\Rightarrow$ item (ii) is trivial, we have to prove only the converse implication. The restriction of $\psi \colon M \to H^\infty(M,\varphi)$ to the $C^*$-subalgebra $C:=B \cap \mathrm{mult}(\psi)$ gives a unital $*$-homomorphism $\rho : C \to H^\infty(M,\varphi)$. Moreover, observe that $J \subseteq C^*(B^\varphi) \subseteq C$ by the first part of Claim \ref{C2}. By the second part of Claim \ref{C2} we have $\varphi (C) \subseteq C$. Since $J$ is simple, $J \cap \mathrm{Ker}(\rho)$ must be $\{0\}$ or $J$.
When $J \cap \mathrm{Ker}(\rho) = J$, $\rho$ factors through $\pi$. Since $\rho(b) = \psi(b) = b$ for all $b \in B^\varphi$ by Eq.\eqref{Eq2}, $\pi$ must be completely isometric on $B^\varphi$, a contradiction to item (ii).
Thus, $J \cap \mathrm{Ker} (\rho) = \{ 0\}$.
For any $b \in B$,
we have $(b-\varphi(b))x \in J$ and moreover $\rho ((b-\varphi(b))x) = \psi(b- \varphi(b)) \rho (x) = 0$ for all $x\in J$, where we used $J \subseteq \mathrm{mult}(\psi)$ and Eq.\eqref{Eq1}.
Hence, $(b-\varphi(b))J = \{0\}$ holds for any $b \in B$.
Since $J$ is an essential ideal of $B$, we conclude that $\varphi = \mathrm{id}$.

\end{proof}

\section{Remarks} 

\subsection{Corollaries of Theorem \ref{T}} Theorem \ref{T} contains the next rigidity property for operator systems generating simple $C^*$-algebras as the particular case when $J = B$.

\begin{corollary} \label{C3} Let $B$ be a unital simple $C^*$-algebra and $\varphi : B \to B$ be a UCP map. Then, $C^*(B^\varphi) = B$ implies $\varphi = \mathrm{id}$.
\end{corollary}

As a corollary of the above corollary we also have:

\begin{corollary} \label{C4} If two UCP maps $\varphi_1 : B_1 \to B_2$, $\varphi_2 : B_2 \to B_1$ between unital simple $C^*$-algebras satisfy that 
both the compositions $\varphi_2\circ\varphi_1 : B_1 \to B_1$ and $\varphi_1\circ\varphi_2 : B_2 \to B_2$ trivially act on some generating sets of $B_1$ and $B_2$, respectively, then $B_1$ is $*$-isomorphic to $B_2$, or more precisely, both $\varphi_1$ and $\varphi_2$ are bijective $*$-homomorphisms with $\varphi_2 = \varphi_1^{-1}$. 
\end{corollary}
\begin{proof} By Corollary \ref{C3} we have $\varphi_2\circ\varphi_1 = \mathrm{id}_{B_1}$ and $\varphi_1\circ\varphi_2 = \mathrm{id}_{B_2}$. In particular, $\varphi_2 = \varphi_1^{-1}$. For every unitary $u \in B_1$ we have $1 = \varphi_2(\varphi_1(u))^*\varphi_2(\varphi_1(u)) \leq \varphi_2(\varphi_1(u)^*\varphi_1(u)) \leq \varphi_2(\varphi_1(u^* u)) = 1$, where the Schwarz inequality is used twice. Hence all the unitaries in $B_1$ fall into the multiplicative domain $\mathrm{mult}(\varphi_1)$. Hence, $\varphi_1$ is indeed a bijective $*$-homomorphism. 
\end{proof}

The reader might think that Corollary \ref{C3} should still hold without assuming that $B$ is simple. However, there is a simple counter-example as follows.
Let $H^2$ be the Hardy space of exponent $2$ over the unit circle $\mathbb{T}$, and $\xi : L^\infty:=L^\infty(\mathbb{T}) \to B(H^2)$ is the map sending each $f \in L^\infty$ to the Toeplitz operator $T_f$ with symbol $f$. Set $B = C^*(\xi(L^\infty))$ and consider the unilateral shift $s := \xi(\chi_1)$ with $\chi_1(z) = z$ for $z\in \mathbb{T}$. It is easy to see that $B(H^2) \ni x \mapsto s^* x s \in B(H^2)$ induces a UCP map $\varphi : B \to B$
and $\xi(L^\infty) \subset B^\varphi$; moreover, these two sets must coincide by Brown--Halmos' theorem. See e.g., \cite[Theorem 4.2.4]{Arveson:Book02}. 
Since $0 \neq 1-ss^* \in B$ and $\varphi(1-ss^*) = 0$, we have $\varphi \neq \mathrm{id}$. Thus, $B = C^*(B^\varphi)$ does not imply $\varphi = \mathrm{id}$ in general without the simplicity of $B$. 

\subsection{Toeplitz operators} Although some part of the discussion below seems implicitly known among specialists, we believe that it is probably worth mentioning it explicitly to make the role of non-commutative Poisson boundaries clear in the context of our discussion. We keep the notations in the last part of \S\S3.1. Observe that $J:=K(H^2)$ sits inside $B=C^*(\xi(L^\infty))$. Hence Theorem \ref{T} implies that \emph{the quotient map $\pi : B \to B/J$ is completely isometric on $B^\varphi = \xi(L^\infty)$}. This is indeed a non-trivial but known fact established in some detailed analysis on Toeplitz operators by Banach algebra technique, see \cite[section 7.15]{Douglas:Book}. Consequently, Arveson's boundary theorem (or Theorem \ref{T} given here) can be regarded as a generalization of this fact. 
Remark that $L^\infty \ni f \mapsto \xi(f) \in B^\varphi$ is known to be isometric (see \cite[Corollary 7.8]{Douglas:Book}) and UCP. Therefore, so is $\pi\circ\xi : L^\infty \to B/J$ too. On the other hand, if $M := B^{**}$ is replaced with $M := B(H^2)$ (to which the $\varphi$ here can be extended as a normal UCP map) in the proof of Theorem \ref{T}, then $M^\varphi = B^\varphi$ and hence $L^\infty \cong H^\infty(M,\varphi)$ by $f \leftrightarrow \xi(f)$ (see the proof of Corollary \ref{C4}). Thus, \emph{the map $\rho$ constructed there in this setting is nothing but the symbol map for Toeplitz operators.} The discussion so far is completely general. Namely, a part of the facts on Toeplitz operators can be generalized as follows. In the situation of the proof of Theorem \ref{T} starting with a normal UCP map $\varphi : M \to M$ instead the normal extension of a given UCP map $\varphi : B \to B$ to $B^{**}$, we have the following observation: 
\emph{\[
\xymatrix{
0 \ar[r] & 
\mathrm{Ker}(\rho) \ar[r] 
&
C^*(M^\varphi) \ar[r]^\rho & 
H^\infty(M,\varphi) \ar[r] 
& 0
}
\] 
is semisplit by the identity map}. As above, \emph{the $*$-homomorphism $\rho : C^*(M^\varphi) \to H^\infty(M,\varphi)$ should be understood as an abstract generalization of the symbol map, and the $C^*$-algebra (or more precisely the von Neumann algebra) structure $H^\infty(M,\varphi)$ gives a view of Toeplitz operators to the fixed-points $M^\varphi$.} 

\subsection{A question} Ken Davidson \cite{Davidson:letter} informed us of another short proof of Arveson's boundary theorem itself by using the existence theorem on maximal dilations for UCP maps. In closing of this note, we would like to ask: \emph{Is there any relation between the concepts of maximal dilations and non-commutative Poisson boundaries ?} 

}

\end{document}